\documentclass[reqno,11pt]{amsart}
\usepackage[margin=1in]{geometry}
\newcommand{\vvNumberWithin}{subsection}

\RequirePackage{amsthm}
\RequirePackage{amsmath}
\RequirePackage{amsfonts}
\RequirePackage{amssymb}
\usepackage[all]{xy}
\RequirePackage{tikz-cd}

\RequirePackage{txfonts}
\RequirePackage{lmodern}
\RequirePackage{enumerate}
\RequirePackage{hyperref}
\RequirePackage{xstring}

\newcommand*{\FIXME}[1]{}

\newcommand*{\Q}{\mathbb{Q}}
\newcommand*{\R}{\mathbb{R}}
\newcommand*{\Z}{\mathbb{Z}}
\newcommand*{\C}{\mathbb{C}}

\DeclareMathOperator{\id}{Id} 
\DeclareMathOperator{\Spec}{Spec}

\let\hom\relax
\DeclareMathOperator{\hom}{Hom}

\newcommand{\VVC}[1]{{}}

\newtheoremstyle{misc}%
     {\topsep}
     {\topsep}
     {}
     {}
     {\itshape}
     {}
     { }
     {}

\newtheoremstyle{newdef}{\topsep}{\topsep}{}{}{\bfseries}{.}{ }{\thmnumber{#2}\thmnote{ #3}}

\newtheorem{master}{Master}[\vvNumberWithin]

\theoremstyle{newdef}

\swapnumbers
\theoremstyle{plain}
\newtheorem{theorem}[master]{Theorem}
\newtheorem*{theorem*}{Theorem}

\newtheorem*{result*}{Result}
\newtheorem{lemma}[master]{Lemma}
\newtheorem*{lemma*}{Lemma}
\newtheorem{corollary}[master]{Corollary}
\newtheorem*{corollary*}{Corollary}
\newtheorem{proposition}[master]{Proposition}
\newtheorem*{proposition*}{Proposition}
\newtheorem{assumption-proposition}[master]{Assumption/Proposition}

\theoremstyle{definition}

\newtheorem*{example*}{Example}

\newtheorem*{application*}{Application}
\newtheorem{definition}[master]{Definition}
\newtheorem*{definition*}{Definition}
\newtheorem{remark}[master]{Remark}
\newtheorem*{remark*}{Remark}
\newtheorem{para}[master]{}
\newtheorem*{para*}{}
\newtheorem{notation}[master]{Notation}
\newtheorem*{notation*}{Notation}

\newtheorem*{question*}{Question}

\newtheorem*{problem*}{Problem}

\theoremstyle{remark}

\newtheorem*{exercise*}{Exercise}

\makeatletter
\@addtoreset{master}{section}
\makeatother

\makeatletter
\@addtoreset{equation}{section}
\makeatother
\numberwithin{equation}{subsection}
 \renewcommand{\theequation}{%
      \ifnum\value{subsection} > 0
      \ifnum\value{subsubsection} > 0
      \thesubsubsection.\Alph{equation}%
      \else%
      \thesubsection.\Alph{equation}%
      \fi%
      \else%
      \thesection.\Alph{equation}%
      \fi%
    }
\newcommand{\nocontentsline}[3]{}
\newcommand{\tocless}[2]{\let\tempcontentsline=\addcontentsline\let\addcontentsline=\nocontentsline#1{#2}\hspace{-1em}\let\addcontentsline=\tempcontentsline}

\newcommand*{\vvspan}[1]{{\langle #1 \rangle}}

\newcommand*{\iso}{\cong}

\title{Motivic intersection complex of certain Shimura varieties}
\author{Vaibhav~Vaish}
\address{Stat-Math Unit, Indian Statistical Institute, 8th Mile, Mysore Road, Bangalore 560059}
\email{vaibhav\_if@isibang.ac.in}

\begin{document}
\begin{abstract} 
Using a version of weight conservativity we demonstrate that for certain Shimura varieties (including all Shimura three-folds, most Shimura four-folds and the Siegel sixfold) the construction of the motivic intersection complex due to Wildeshaus compares with a motivic weight truncation in the sense of S.~Morel. In particular it is defined up to a unique isomorphism, and satisfies the intrinsic characterization for an intermediate extension due to Wildeshaus.
\end{abstract}
\maketitle
\section{Introduction}
\subsection{}
   The purpose of this article is to complete and extend the results in \cite{vaish2017weight} in the particular context of Shimura varieties, and in particular construct the intersection motive for the Baily-Borel compactification of certain Shimura varieties satisfying the intrinsic characterization of Wildeshaus \cite{wildeshaus_ic}.\\
   
      Recall that in \cite[4.2.7]{vaish2017weight} one could define a canonical intersection motive for an arbitrary threefold $X$ as an object in the triangulated category of mixed motivic sheaves, $IM_X\in DM(X,\Q)$ (see \ref{intro:dmx}). One of the key limitations there was that we could not show that $IM_X$ satisfies the intrinsic characterization of an intersection motive as defined in Wildeshaus \cite{wildeshaus_ic}, or, even more elementarily, show that it is a relative Chow motive (that is of weight $0$ in the sense of Bondarko \cite{bondarko2014weights} or Hebert \cite{hebert2011structure}, or, equivalently due to Fangzhou \cite{fangzhou2016borel} lives in the full subcategory of relative Chow motives $CHM(X)$ of Corti-Hanamura \cite{cortiHanamura}). 
   
   On the other hand in \cite{wildeshaus_shimura_2012} Wildeshaus constructs an intersection motive for arbitrary Shimura varieties, satisfying a slightly weaker characterization \cite[Definition 2.10]{wildeshaus_shimura_2012} -- the key difference being that the intersection motive is not required to be defined canonically, but only ``upto radical''; however it is indeed a relative Chow motive.\\
   
   In this article, we reconcile the results of \cite{vaish2017weight} and \cite{wildeshaus_shimura_2012} where they are relevant and in particular show that in the case of Shimura three-folds, the two constructions are the same. Thus $IM_X$ is both a Chow motive and defined canonically and more generally satisfies the internal characterization of \cite{wildeshaus_ic}. Our results are also applicable in specific higher dimensional cases, for example, for most Shimura fourfolds and the Siegel sixfold.
  
 \subsection{}  
   The role of intersection motive is not merely technical. The intersection cohomology of the Baily-Borel compactification of a Shimura variety is supposed to contain useful arithmetic information -- for example, it is the natural playing ground for Galois representations associated to automorphic forms. For this reason it is useful to have a canonical intersection motive (which plays well with the Hecke operators) and not merely ``upto radical''.  
   
    However, to be useful in this context, we also need to construct the intersection motive for more than just the constant local system. For example, over Siegel sixfold, one would expect to construct the intersection motive of the Kuga-Sato families relative to its Baily-Borel compactification, and one hopes to construct the same over $\Q$, the reflex field. In general, such constructions have been of interest and several results are known -- the case of modular curves is classical \cite{scholl1994classical}, more recently we have the case of Hilbert-Blumenthal varieties (non-constant coefficients) \cite{wildeshaus2009interior} and (constant coefficients) \cite{vaish2016motivic}, Picard modular surfaces \cite{wildeshaus2015interior}, Picard modular varieties in any dimension \cite{cloitre2017interior} and most recently, the case of Siegel threefold \cite{wildeshaus2017intersection}.
       
   While the methods in the present article are not sufficient for a complete solution here, we also demonstrate construction of certain intersection motives with non-constant local systems for the Siegel sixfold.
   
 \subsection{}
We now describe the main results of the article in some detail. We begin with the main Theorem \ref{tech:mainresult} which is slightly technical. For the purpose of this article we will only work with schemes that are of finite type over a field $k$ of characteristic $0$ -- the limitation on characteristic is not entirely necessary, but our main applications of interest are in the context when $k$ is the reflex field of a Shimura variety.

Call a proper morphism $\pi:Y\rightarrow X$ to be of Abelian type of relative dimension $\le d$ if, roughly,  over each Zariski point of $X$ the motive of the fibre lives in the smallest triangulated category generated by motives of Abelian varieties of dimension $\le d$ upto arbitrary Tate twists (see \ref{def:abtype} for a precise definition). This is closely related to the motives of Abelian type of \cite[7.13]{wildeshaus_shimura_2012} but stated in a form that simplifies calculations for us.

For $d\le 2$, enough motivic cohomology calculations are known in this subcategory to yield the following:

\begin{theorem*}[See \ref{tech:mainresult}]
	For a variety $X$, let $\pi: Y\rightarrow X$ be a proper morphism with $Y$ smooth such that for some $V\subset X$ with $W=\pi^{-1}V$, $\pi|_{W}$ is an abelian scheme and $\pi|_{Y-W}$ is of Abelian type of relative dimension $\le 2$. Then for any summand $N$ of $\pi_*1_W\in DM(V)$ which realizes to a local system (upto shifts), the intermediate extension $j_{!*}N\in DM(X)$ is defined canonically as a Chow motive and satisfies the intrinsic characterization of \cite{wildeshaus_ic} (see \ref{richerIC}).
\end{theorem*}

As stated earlier this result is a consequence of reconciliation of the corresponding calculations in \cite{vaish2017weight} (which is where the restriction of $d\le 2$ comes from and where the motivic cohomology calculations were used) and \cite{wildeshaus_shimura_2012} (which is where the restriction to Abelian type comes from). \\

This general result yields to more specific results for Shimura varieties (see \ref{shimura:situation} for precise definitions):

\begin{corollary*}[See \ref{shimura3folds}] 
	Let $X$ be any arbitrary Shimura threefold, or the Siegel sixfold (defined over it's reflex field $k$) and let $Y:=X^*$ denote it's Baily-Borel compactification. Then, the intersection motive $IM_Y$ exists in $DM(Y)$ in the sense of \cite{wildeshaus_ic} (see \ref{richerIC}). 
\end{corollary*}

Note that, the methods here also work for most Shimura fourfolds, see \ref{result:4folds}. As stated earlier, for the Siegel sixfold, we can even work with certain local systems:
\begin{corollary*}[See \ref{siegel6}] 
	Let $X$ be the Siegel sixfold (defined over it's reflex field $k$) and let $Y:=X^*$ denote it's Baily-Borel compactification. Let $\pi:A\rightarrow X$ denote the universal abelian scheme over $X$ (which exists, for appropriate choice of arithmetic subgroups). Then, the intersection motive $IM_Y(N) := j_{!*}N$ exists in $DM(Y)$ in the sense of \cite{wildeshaus_ic} (see \ref{richerIC}) where $N$ is any summand of $\pi_*1_A$ which realizes to a local system (upto shifts).
\end{corollary*}

\subsection{} The main method in this article is motivated by \cite{wildeshaus_conservativity} of Wildeshaus -- we use conservativity of the realization functors restricted to the triangulated category generated by motives of Abelian varieties (and arbitrary Tate twists) to calculate weights. However, for our purpose, it is not sufficient to work with weights in the sense of Bondarko \cite{bondarko2014weights} or Hebert \cite{hebert2011structure} (which are motivic version of weights in the sense of Deligne), but we need to work with the motivic analogue of constructions due to S.~Morel \cite[\S 3]{morelThesis} (or more precisely, the mild generalization in \cite[\S 3]{arvindVaish}). We briefly motivate the method below.

Recall that for appropriate fields $k$, any variety $X/k$ and any monotone step function $F$, in \cite[3.1.7]{arvindVaish} one defines a pair of subcategories $({^wD^{\le F}(X)}, {^wD^{> F}(X)})$ of the appropriate category of mixed sheaves over $X$ denoted there as $D(X)$ (e.g. Deligne's category of mixed $l$-adic sheaves $D^b_m(X, \Q_l)$ \cite{BBD} for $k$ finite or the derived category of mixed Hodge modules of Saito $D^bMHM(X)$ \cite{saito_mhm_1990} when $k=\C$). These subcategories form both a weight structure and a $t$-structure on $D^b(X)$ and are inspired by S.~Morel's construction \cite[\S 3]{morelThesis}, where she also demonstrates the fundamental relation (see \cite[3.4.2]{morelThesis}):
\[
	j_{!*}(\mathcal L[d]) = w_{\le n+d}j_*\mathcal L[d] \longrightarrow j_*(\mathcal L[d])
\]
where $\mathcal L$ is a local system of weight $n$ on any $j:U\subset X$ regular, $X$ equidimensional of dimension $d$, and $w_{\le n+d}$ denotes the truncation for the Morel's $t$-structure for the constant function $F=n+d$. 

In \cite{vaish2017weight} we constructed motivic analogue of this $t$-structure and weight structure, that is subcategories $({^wDM^{\le F}(X)}, {^wDM^{> F}(X)})$, inside appropriate subcategories of $DM(X)$, for $F=\id, \id+1$ and $F=(3\mapsto 3, 2\mapsto 3, 1\mapsto 2, 0\mapsto 2)$ (in fact, as we will show here, we could have worked with $F$ arbitrary provided we restrict to the triangulated category generated by motives of surfaces upto arbitrary Tate twists). This allows us to recover the intermediate extension as a weight truncation in the motivic context.

If we restrict to Chow motives of abelian type, for such a motive $N$, the intermediate extension $j_{!*}N$ has also been constructed by Wildeshaus \cite[0.1]{wildeshaus_shimura_2012}, as a weight truncation (in the sense of Bondarko) of $j_*N$ which is minimal (see \cite[2.2]{wildeshaus_shimura_2012} for a definition of minimality) yielding naturally a map:
\[
	j_{!*}N \longrightarrow j_*N
\]

 The key property here that we are interested in is that this $j_{!*}N$ realizes to the intersection complex in the category of constructible $\Q_l$ sheaves due to \cite{ekedahlAdic}. While this is not a mixed category, realizations of morphisms here play well with the weight filtrations \cite[2.1.2]{bondarko_weights} allowing us to recover milder analogues of S.~Morel's construction. 
 
 Since we are restricting to the category of motives of Abelian type, following \cite{wildeshaus_conservativity} we observe that the realizations here are conservative; and in particular we use the known calculations about weights in the realizations to conclude that 
 \begin{align*}
 	j_{!*}N\in {^wDM^{\le F}}(X)& &Cone(j_{!*}N\longrightarrow j_*N)\in  {^wDM^{> F}}(X)
 \end{align*}
for appropriate $F$. This is enough to conclude that the two calculations of intermediate extension are the same, thereby leading to our main result.

\subsection{Outline} 
	Section \S \ref{sec:prelim} contains preliminaries which are well known in literature or are implicit elsewhere. In \S \ref{sec:gluing} we talk about a variant of gluing of $m$-structures (by which we mean a pair of categories which is both a $t$-structure and a weight structure, see \ref{def:mstruct}) -- the useful result is \ref{gluing:mainresult} which basically says that the full subcategory of objects for which (appropriate) truncations are defined form a triangulated subcategory and the corresponding truncations then come from an $m$-structure. In \S \ref{sec:newmotives} we discuss motivic sheaves briefly, with the purpose of fixing the key notations. We also briefly discuss realizations in \ref{discuss:realizations} and even though the category in which the realization functors map to is not mixed, we define an analogue of the subcategories $^wD^{\le n}$ (resp. $^wD^{>n}$) of S.~Morel in \ref{def:weightTruncationInRealization}. In \S \ref{sec:chowForAb} we recall Chow-Kunneth decomposition for abelian varieties while in \S \ref{sec:motivic_ic} we discuss (motivic) intersection complex and weights on the same.

	Section \S \ref{sec:mainresults} contains the main technical results. In \S \ref{sec:overField} we discuss the analogue of Morel's weight truncations over a field. The main result is \ref{tstruct:DMAb}, which is already implicit in \cite{vaish2017weight} though not stated explicitly, and which states that motivic analogue of Morel's $m$-structure can be defined on the triangulated category generated by abelian varieties of dimension $\le 2$ upto arbitrary Tate twists. In \S \ref{sec:spread} we relativize this construction. In particular we prove a conservativity principle \ref{weightsOfRealization} and it's converse \ref{spreading:truncation} eventually leading to the main theorem \ref{tech:mainresult}. In the proof, instead of working with the formalism of gluing for motives of abelian type \ref{def:abtype} we use the simpler result \ref{abelianTypeRestrictWell}.
	
	Finally in \S \ref{sec:shimura} we summarize some consequences of our result in the context of Shimura varieties. After briefly discussing the situation in \ref{shimura:situation} we state our main result in \ref{shimura:tech:mainresult} which has consequences for all Shimura varieties of $\dim \le 3$ (see \ref{shimura3folds}) as well as for the Siegel six fold (see \ref{siegel6} and \ref{siegel6local}).

\subsection{Acknowledgement} This research was supported by the grant under INSPIRE fellowship scheme, DST/INSPIRE/04/2015/000120 and the hospitable atmosphere at the Indian Statistical Institute, Bangalore. We would also like to thank A.~Nair and Wildeshaus for useful discussions.

\subsection{Notation} All schemes $X$ will be separated of finite type over a base field $k$ which will be assumed to be of characteristic $0$. For a scheme $X$, $X_{red}$ denotes the underlying reduced scheme. By $\Spec k\hookrightarrow X$ we mean a Zariski point $x=\Spec k$ in $X$. A locally closed $Z\subset X$ will always be given the reduced induced sub-scheme structure.
\section{Preliminaries}\label{sec:prelim}
\subsection{Gluing}\label{sec:gluing}%

\begin{definition}\label{def:mstruct}
A \emph{$t$-structure} on a triangulated category $D$ (see \cite{BBD}) is a pair of full subcategories $(D^{\le }, D^{>})$ satisfying three properties:
\begin{itemize}
\item (Orthogonality) $\hom (a,b)=0, \forall a\in D^{\le }, b\in D^{>}$
\item (Invariance) $D^{\le }[1]\subset D^{\le }$, and $D^{> }[-1]\subset D^{>}$
\item (Decomposition) $\forall a\in D$, there is a distinguished triangle $a_{\le}\rightarrow a\rightarrow a_>\rightarrow $ with $a_{\le}\in D^{\le }$ and $a_>\in D^{>}$. 
\end{itemize}

A \emph{weight structure} (also called a co-$t$-structure) on a triangulated category $D$ (see \cite{bondarko_weights}, for example) is the same as a $t$-structure, except that instead of invariance, it satisfies co-invariance, and we need an additional condition for closure under summands (which is automatic for $t$-structures):
\begin{itemize}
	\item (Karaubi-closed) $D^{\le }$ and $D^{>}$ are closed under taking summands.
	\item (Orthogonality) $\hom (a,b)=0, \forall a\in D^{\le }, b\in D^{>}$
	\item (Co-invariance) $D^{\le }[-1]\subset D^{\le }$, and $D^{> }[1]\subset D^{>}$
	\item (Decomposition) $\forall a\in D$, there is a distinguished triangle $a_{\le}\rightarrow a\rightarrow a_>\rightarrow $ with $a_{\le}\in D^{\le }$ and $a_>\in D^{>}$. 
\end{itemize}

	We define an \emph{$m$-structure} to be a pair of full subcategories which form both a $t$-structure and a weight structure. In particular, for an $m$-structure $D^{\le }$ and $D^>$ are triangulated subcategories and $a\mapsto a_{\le}$ as well as $a\mapsto a_>$ are triangulated functors. 
\end{definition}

\begin{definition}
	Let $D$ be a triangulated category, and $S\subset D$ be a collection of objects of $D$. We define $\vvspan{S}$, the span of $S$, to be the smallest triangulated subcategory of $D$ containing $S$ which is closed under taking summands. We do not insist $\vvspan S$ to be closed under arbitrary direct sums. 
	
	The objects of $\vvspan S$ can be constructed by (finitely many iterations of) taking shifts, extensions and summands of objects of $S$. 
\end{definition}

Then we have the following proposition which can be proved by an easy induction.
\begin{proposition}\label{tFromGenerators}
	Let $A,B,H\subset D$ be a collection of objects of a triangulated category $D$. Assume $\hom(A,B[n])=0$ for all $n\in \Z$, and
	\[
		h\in H\Rightarrow \exists a\in A, b\in B\text{ such that there is}\text{ a distinguished triangle }a\rightarrow h\rightarrow b\rightarrow.
	\]
	Then if  $\vvspan A, \vvspan B\subset \vvspan H$, the pair $(\vvspan A , \vvspan B )$ is a $m$-structure on the triangulated subcategory $\vvspan H $.
\end{proposition}
\begin{proof}
	See \cite[2.1.4]{vaish2016motivic}. 
\end{proof}

\begin{para}\label{gluing:situation}
	We fix a scheme $X$ (which will quickly be assumed to be Noetherian of finite type over a field $k$) and work with (versions of) derived category of (motivic) sheaves on sub-schemes $W$ of $X$. More generally, we assume that for each $W\hookrightarrow X$ we are given a triangulated subcategory $D_W$, and for each $f:W\hookrightarrow W'$, adjoint pairs:
	\begin{align*}
		f^*:D_{W'}\leftrightarrows D_W:f_* & & f_!:D_W\leftrightarrows D_{W'}:f^!
	\end{align*}
	which satisfy the formalism of Grothendieck's four functors. We also assume that for each (Zariski) point $\epsilon:\Spec L\hookrightarrow X$, we have a triangulated category $D(L)$ such that continuity holds (see \cite[\S 3]{vaish2017punctual} for a detailed summary). In particular there is a pullback $\epsilon^* : D_X\rightarrow D(L)$. If $Y$ denotes the closure of $\epsilon$ in $X$, we define:
	\begin{align*}
		\epsilon^! := \epsilon_Y^*f^!:D_X\rightarrow D(L) & &\text{ where }& &\epsilon:\Spec L \overset{\epsilon_Y}\hookrightarrow Y\overset f\hookrightarrow X\text{ is the natural factorization.}
	\end{align*}
\end{para}

\begin{definition}\label{gluing:spreadingOut}
		Assume the situation of \ref{gluing:situation} on a scheme $X$. Further assume that for each $\Spec k\hookrightarrow X$, we are given a $m$-structure $(D^{\le}(k), D^{>}(k))$ on a full subcategory $D'(k)\subset D(k)$. 
		For any $U\hookrightarrow X$, define
		\begin{align*}
			D^{\le}(U)	:=\{a&\in D_U\big| \epsilon^*(a)\in D^{\le}(k)\text{ for }\epsilon:\Spec k\rightarrow U\text{ any point of }U\} \\
			D^{>}(U)	:=\{a&\in D_U\big| \epsilon^!(a)\in D^{>}(k)\text{ for }\epsilon:\Spec k\rightarrow U\text{ any point of }U\}\\
			D'(U) := \{a&\in D_U\big| \exists b\rightarrow a\rightarrow c\rightarrow \text{ s.t. }b\in D^{\le}(U), c\in D^{>}(U)\}
		\end{align*}
		as full subcategories. In particular if $f:S\hookrightarrow T$ is an immersion, $S, T\in Sub(X)$:
		\begin{align*}
			f^*(D^{\le}(T))\subset D^{\le}(S)& &f^!(D^{>}(T))\subset D^{>}(S)
		\end{align*}
\end{definition}

	We then have the following simple minded gluing:
	\begin{proposition}\label{gluing:mainresult}
		Assume the situation of \ref{gluing:spreadingOut} on a Noetherian scheme $X$. Then $D'(X)$ is a pseudo-abelian triangulated subcategory of $D(X)$ and the pair $(D^{\le}(X), D^{>}(X))$ forms an $m$-structure on $D'(X)$.
	\end{proposition}		
	\begin{proof}
			(Orthogonality of $D^{\le}(U)$ and $D^{>}(U)$) Let $a\in D^{\le}(X)$ and $b\in D^{>}(X)$. We do a Noetherian induction on $X$. The base case is when $X=\Spec k$, in which case it follows since $D^{\le}(k)\perp D^{>}(k)$. Let $\eta:\Spec K\rightarrow X$ be a generic point of $X$. Then $\hom(\eta^*(a), \eta^*(b)) = 0$ by definition, using  $D^{\le}(K)\perp D^{>}(K)$ (notice that $\eta^!=\eta^*$ in this case, by definition). Let $f\in \hom(a,b)$, then since $\eta^*(f)=0$, by continuity there is a $j:U\subset X$ open, such that $j^*(f)=0$. Since $j^*j_* = \id$, it follows that the composite: $a\xrightarrow{f} b \rightarrow j_*j^*b$ vanishes. Hence $f$ factors through the third term in the triangle, $i_*i^!b$, where $i:Z=X-U\hookrightarrow X$ is the complement. But $\hom(a,i_*i^!b)=\hom(i^*a, i^!b)=0$ by induction hypothesis and we are done.

			($D^{\le}(U)$ and $D^{>}(U)$ are pseudo-abelian triangulated subcategories) Since $f^*$, $f^!$ commute with shifts, cones, and taking summands this is immediate. 
									
			($D'(X)$ is closed under shifts): Let $a\in D'(X)$. Then if $b\rightarrow a\rightarrow c\rightarrow $ is a decomposition with $b\in D^\le (X)$, $c\in D^{>}X$, the $b[m]\rightarrow a[m]\rightarrow c[m]\rightarrow $ implies that $a[m]\in D'(X)$ since $D^{\le}(U)$ and $D^{>}(U)$ are invariant under shifts.
			
			($D'(X)$ is closed under cones): Let $a,a'\in D'(X)$. Then $b\rightarrow a\rightarrow c\rightarrow $ and $b'\rightarrow a'\rightarrow c'\rightarrow $  is a decomposition with $b,b'\in D^\le (X)$, $c,c'\in D^{>}X$. Now for any morphism $f:a'\rightarrow a$, consider the diagram:
			\[\begin{tikzcd}
					b'	\ar[d,"f_b", dotted]\ar[r]	&a'	\ar[r]\ar[d,"f"]			&c'\ar[d,"f_c", dotted]\ar[r]	&\-\\
					b	\ar[r]\ar[d]				&a	\ar[r]\ar[d]				&c	\ar[r]\ar[d]		&\-\\
					Cone(f_b)\ar[r]\ar[d]			&Cone(f)	\ar[r]\ar[d]		&Cone(f_c)\ar[r]\ar[d]		&\-\\
					\-								&\-								&\-						&\-
			\end{tikzcd}\]
			where $f_b, f_c$ exist since $\hom(b',c)=0$. $Cone(f_b)\in D^\le (X)$, and $Cone (f_c)\in D^>(X)$ since this can be tested after applying $\epsilon^*$ (resp. $\epsilon^!$) for any point $\epsilon: \Spec k\hookrightarrow X$. It follows that $D'(X)$ is closed under taking cones. 
			
			($D'(X)$ is closed under summands)
				Suppose $a'\in D(X)$ is a summand of $a\in D'(X)$. There is a projector $p:a\rightarrow a$ s.t. $a'$ is an image of $a$. There is a triangle: 
		\[
			b\overset f\longrightarrow a \longrightarrow c \longrightarrow 
		\]
		 with $b\in D^{\le}(X)$ and $c\in D^{>}(X)$. Using orthogonality it is easy to see that $p$ gives rise to a projector $p_b:b\rightarrow b$. Let the kernel of $p_b$ be $b'$ and that for the projector $1-p_b,1-p$ be $b'',a''$ respectively. Let $f$ induce maps $f':b'\rightarrow a'$ and $f'':b''\rightarrow a''$. Then we get an induced morphism of distinguished triangles:
		 \[\begin{tikzcd}
			&b'\oplus b''\ar[r, "f'\oplus f''"]\ar[d]	&a'\oplus a''\ar[r]	\ar[d]	&c'\oplus c''\ar[r]\ar[d]	&\-		\\		 
		 	&b\ar[r,"f"] 								&a\ar[r]					&c\ar[r]					&\-					
		 \end{tikzcd}\]
		 where $c',c''$ are defined as cones of the morphism $f',f''$. Then first two vertical maps are isomorphisms, and hence so is the third. In particular $c'$ is a summand of $c$ and hence in $D^{>}(X)$. Since $b'$ is in $D^{\le}(X)$ and $b'\rightarrow a'\rightarrow c'\rightarrow$ is a distinguished triangle, we are done.
	\end{proof}

\begin{remark}
	In using the previous proposition, the hard task would be to determine $D'(X)$ explicitly. It does not seem likely that the categories $D'(U)$ would automatically satisfy the formalism of four functors. We will be completely bypassing the question of determining $D'(X)$ in our cases of interest and instead work with objects which will be known to be in $D'(X)$.
\end{remark}

\subsection{Motivic Sheaves}\label{sec:newmotives}

\begin{para} \label{intro:dmx}
Given any base scheme $S$, there exists a rigid tensor triangulated category of motivic sheaves $DM(S)$ with unit object denoted $1_S$, Tate twists denoted $A\mapsto A(r)$ and such that the formalism of Grothendieck's six functors holds. 

One choice for such a construction is the category of motivic sheaves without transfers as constructed by Ayoub in \cite{ayoub_thesis_1, ayoub_thesis_2}. This is the category $\mathbb{SH}_\mathfrak{M}^T(S)$ of \cite[4.5.21]{ayoub_thesis_2} with $\mathfrak{M}$ being the complex of $\Q$-vector spaces (and one works with the topology \'etale topology), also denoted as $DA(S)$ in the discussion \cite[2.1]{ayoub2012relative}. To play well with the realization functors and continuity, we will instead restrict attention to the subcategory of compact objects in $DA(S)$, which are also stable under the Grothendieck's four functors. The second choice for the construction is the compact objects in the category of motivic sheaves with transfers, the Beilinson motives $DM_{B,c}(S)$ as described in the article \cite{cisinski2012triangulated}. Again, the objects in $DM_{B,c}(S)$ are compact by construction and play well with realizations.

We refer the reader to \cite[\S 2.6]{vaish2017weight} for a summary of properties of the category (except about realizations) which will be used here.
\end{para}
\begin{para}
Over a field $S=\Spec k$, the construction of $DM(\Spec k)$ is due to Voevodsky \cite{voevodsky}, who also shows that there are natural functors:
	\begin{equation*}
			(SmProj/k)^{op}\overset {h_k} \longrightarrow CHM(k)\hookrightarrow	DM(\Spec k)
	\end{equation*}
	where $SmProj$ denotes the category of smooth projective varieties over $\Spec k$, $CHM(k)$ denotes the category of Chow motives over the field $k$, and the last functor is fully faithful. This can be extended to give a functor:
	\[
		h_k: (Sch/k)^{op} \rightarrow DM(\Spec k)
	\]
	More generally, due to work of Corti-Hanamura \cite{cortiHanamura}, there is a category of Chow motives $CHM(S)$ over any regular base $S$. Due to work of Hebert \cite{hebert2011structure} and Bondarko \cite{bondarko2014weights} there is a weight structure $(DM_{w\le 0}(S), DM_{w>0}(S))$ on $DM(S)$ and due to work of Fangzhou \cite{fangzhou2016borel}, it's heart can be identified with the category of Chow motives. Therefore we have functors
	\begin{equation*}
		(SmProp/S)^{op}\overset {h_S} \longrightarrow CHM(S)	\hookrightarrow DM(S) \text{ with essential image as heart }DM_{w\le 0}(S)\cap DM_{w\ge 0}(S)
	\end{equation*}
	where $SmProj/S$ is the category of proper schemes over $S$ which are smooth, and the last functor is fully faithful. This can be extended to give a functor:
	\[
		h_S: (Sch/S)^{op} \rightarrow DM(S)	\text{ with }(\pi:X\rightarrow S)\mapsto \pi_*1_X
	\]
\end{para}

\begin{para}
	As a matter of notation, we will often identify the objects in $CHM(S)$ and $DM(S)$ via the above fully faithful embedding. We will often write $h$ for $h_S$ when the base scheme is clear from the context.
\end{para}

\begin{para}[$l$-adic realization and weight filtrations]\label{discuss:realizations} 
	Let $k$ be of characteristic $0$ and $l$ any prime, then we have  monoidal realization functors (at least for $DM(S) = DM_{B,c}(S)$, see \cite[\S 7.2]{cisinski2016etale})
	\[
		r_{l, S}:DM(S) \rightarrow D^b_c(S) 
	\]
	whose target is the (bounded, derived) category of costructible $\Q_l$ sheaves due to \cite{ekedahlAdic}. The realizations commute with Grothendieck's four functors (\cite[7.2.24]{cisinski2016etale}). 
	
	We will denote the realization functors (for any $l$) by $r$. The target category $D^b_c(S)$ admits a perverse $t$-structure and this gives rise to a homological realization functor:
	\[
		^pH^*r = ({^pH^i}\circ r)_{i\in \Z}: DM(S) \rightarrow Gr_\Z Perv(S)
	\] 
	where $Perv(S)$ denotes the heart of the perverse $t$-structure. 
	The target category for the realization functor, $D^b_c(S)$ does not admit any natural weight structure. However, if $k$ is finitely generated over $\Q$ and we are looking at the $l$-adic realization, each object (in the image of the realization) does admit a weight filtration (\cite[2.5.1]{bondarkoMixedMotivicSheaves}). 
	
	In any case, due to \cite[2.1.2]{bondarko_weights}, there is a canonical weight filtration on $^pH^i(r(A))$ induced by taking any choice of Bondarko's weight truncations in $DM(S)$, and this coincides with the above weight filtration if $k$ is finitely generated over $\Q$. 
	
	In particular, for any morphism $f:A\rightarrow A'$ in $DM(S)$ the realization $^pH^i(r(f))$ is strict for the corresponding weight filtrations (since $S$ is assumed to be of finite type over $k$ it is very reasonable in the sense of \cite[2.1.1]{bondarkoMixedMotivicSheaves}, and then we use \cite[2.5.4(II)(1) and 1.3.2(II)(2)]{bondarkoMixedMotivicSheaves}).
\end{para}
	
	This allows us to make the following definition, motivated by \cite[\S 3]{morelThesis}:

\begin{definition}\label{def:weightTruncationInRealization}
	In below, assume that $A\in D^b_c(S)$ and $^pH^i(A)$ has a weight filtration by assuming that $A$ is in the image of the realization functor $r$. Then, for any $n\in \Z$ define the subcategories:
	\begin{align*}
		^wD^{\le n}&:=\{A\in D^b_c(S)\big| A=r(A'), {^pH^j}(A)\text{ has weights}\le n\} \\ {^wD^{> n}}&:=\{A\in D^b_c(S)\big| A=r(A'), {^pH^j}(A)\text{ has weights}> n\}
	\end{align*}
	The definitions here follow the definitions in \cite[\S 3]{morelThesis}, however note that $D^b_c(S)$ is not mixed, and hence this does not form a $t$-structure on $D^b_c(S)$.
\end{definition}

\subsection{Chow-Kunneth decomposition for Abelian varieties}\label{sec:chowForAb}
We recall the Chow-Kunneth decomposition of Abelian varieties over a base field $k$, or more generally, over a base scheme $S$. A priori the decomposition is constructed in $CHM(S)$ (or rather, as in \cite{deninger1991motivic}, in a full subcategory of the same), but this can be thought of as a decomposition in $DM(S)$ using the embedding $CHM(S)\hookrightarrow DM(S)$. 

We have the following main result:

\begin{proposition}
	Let $\pi:A\rightarrow S$ be an abelian scheme of relative dimension $d$ over a regular base $S$. Then we have the following canonical decomposition (Chow-Kunneth decomposition) of $h(A)$ in $CHM(S)\hookrightarrow DM(S)$:
	\begin{align*}
		h_S(A) = \bigoplus_{0\le i\le 2d} h^i_S(A)
	\end{align*}
	such that there are natural identifications (Poincare duality)
	\begin{align*}
		h^{d-i}_S(A)(-i)[-2i] = h^{d+i}_S(A) \text{ whenever } 0\le i\le d.
	\end{align*}
	Here $h^i_S(A)$ can be characterized as the part of the motive $h_S(A)$ where $[\times n]^*$ acts by multiplication by $n^i$ (where $[\times n]$ is the power $n$ operation on the abelian group scheme $A/S$). In particular, if $f:T\subset S$ is regular of codimension $c$, we have a natural identification in $DM(T)$:
	\begin{align*}
		f^*(h_S^i(A)) = h_T^i(A|_T)	&	& f^!(h_S^i(A)) = f^*(h_S^i(A)) (-c)[-2c] = h_T^i(A|_T)(-c)[-2c]
	\end{align*}
\end{proposition}
\begin{proof}
	See \cite[3.1]{deninger1991motivic} for definition of $h^i_S(A)$ and it's  characterization via action of $[\times n]^*$. Poincare duality follows from the last remark following \cite[3.2]{deninger1991motivic}.
	
	Now $f^*(h^i_S(A)) = h^i_T(A|_T)$ follows since the power $n$ operation on $A$ restricts to the same on $A|_T$. Furthermore $f^!(h_S^i(A)) = f^*(h_S^i(A))(-c)[-2c]$ by purity. 
\end{proof}

\begin{definition}
	Let $\pi:A\rightarrow S$ be an abelian scheme of relative dimension $d$. For $i<0$ or $i>2d$ it would be convenient to define $h^i_S(A)=0$. Then for any $i\in \Z$ we can define:
	\begin{align*}
		h_S^{\le j}(A) := \oplus_{i\le j}h_S^i(A)& &h_S^{> j}(A) := \oplus_{i> j}h_S^i(A)
	\end{align*}
	as sub-objects of $h_S(A)$.
\end{definition}

\begin{remark}\label{h:oldisnew}
	Objects $h^i_k(X)$ for $i=0,1$ can be defined for arbitrary smooth projective varieties over $k$ and for $i=2, 3, 4$ for any smooth projective surface over $k$ as in \cite[\S 2.4]{vaish2017weight}; by \cite[5.3]{scholl1994classical} it is equivalent to the definitions above.
\end{remark}

Finally we recall the following result from \cite{wildeshaus_conservativity}:
\begin{proposition}[Conservativity]\label{conservativityonab}
	Let $S$ be regular of finite type over $k$. Then the restriction of the realization functor to the category 
	\[
		DM^{Ab}(S) := \vvspan{h_S(X)(j)\big| X/S\text{ an abelian scheme}, j\in \Z}
	\]
	is conservative.
\end{proposition}
\begin{proof}
	By \cite[Definition 2.5(b)]{wildeshaus_conservativity}, this is a subcategory of the category of Abelian motives over $S$ for trivial stratification (denoted there as $DM^{Ab}_{B,c}(S)_\Q$). Then \cite[Theorem 3.3(b)]{wildeshaus_conservativity} tells us that the realization functor on $DM^{Ab}(S)$ is conservative.
\end{proof}

\subsection{(Motivic) Intersection complex}\label{sec:motivic_ic}
\begin{para}[Motivic Intersection complex upto radical]\label{def:motivicIC}
	Let $X$ be an arbitrary variety over $k$, and $j:U\hookrightarrow X$ dense immersion with $U$ open regular. For any object $N\in DM(U)$ such that $N$ is of weight $0$ (in the sense of Hebert \cite{hebert2011structure} or Bondarko \cite{bondarko2014weights}), consider the extension $j_{!*}N\in DM(X)$, as defined in \cite[Definition 2.10]{wildeshaus_shimura_2012} (which is a slightly weaker notion than an unconditional definition in \cite{wildeshaus_ic} but exists unconditionally in more cases, for example, the Baily-Borel compactification of an arbitrary Shimura variety). The main interest in the motivic intersection complex comes from the fact that its realization corresponds to the ordinary intersection complex.
 
 	The object $j_{!*}N$ (called as the \emph{middle extension} of the motive $N$ or the \emph{motivic intersection complex for the motive $N$}) is unique up to isomorphism, and we will denote $j_{!*}1_U$ by $IM_X$ (referred to as \emph{the motivic intersection complex}). It is shown in \cite[Theorem 0.1]{wildeshaus_shimura_2012} that we have the following properties:
 	\begin{enumerate}
		\item There is no summand of $j_{!*}N$ of the form $i_*L_Z$ with $L_Z$ of weight $0$ in $DM(Z)$, where $i:Z\hookrightarrow X$ denotes the closed complement of $U$.
		\item There is an isomorphism $j^*j_{!*}N \iso N$. The induced map $End(j_{!*}N)\rightarrow End(N)$ has a nilpotent kernel 
		\item $(1)$ and $(2)$ characterize $j_{!*}N$ up to (a non-unique) isomorphism.
		\item The following analogue of the decomposition theorem holds: let $M \in DM(X)$ be of weight $0$ such that $j^*M \iso N$ for some smooth open dense $j:U\hookrightarrow X$. Then (if $j_{!*}N$ exists) there is a non-canonical isomorphism:
	\[
		M\overset \iso \longrightarrow j_{!*}N \oplus i_* L_Z
	\]		
	where $i:Z\hookrightarrow X$ is a proper closed immersion and $L_Z\in DM(Z)$ is of weight $0$.
	\end{enumerate}
\end{para}

We have the following simple strengthening of $(4)$ which will be useful:
\begin{lemma}\label{ICissummandofchow}
	Let $j:U\hookrightarrow X$ denote be an open immersion. Let $N$ be of weight $0$ in $DM(U)$ and assume that the motivic intersection complex $j_{!*}N$ exists. Assume that $M$ is of weight $0$ in $DM(X)$ such that $N$ is a summand of $j^*M$. Then $j_{!*}N$ is a summand of $M$.
\end{lemma}
\begin{proof}
	By \cite[1.7]{wildeshaus_ic} the functor $j^*$ is full on $DM(X)_{w=0}$. Since $N$ is a summand of $j^*M$, it follows that there are morphisms
	$
		id:N\overset\alpha\longrightarrow j^*M\overset\beta\longrightarrow N
	$
	and in turn they lift to morphisms:
	$
		r:j_{!*}N\overset{\bar\alpha}\longrightarrow M\overset{\bar\beta}\longrightarrow j_{!*}N
	$
	where the composite $r = \bar\beta\bar\alpha$ with $1-r \in K:=Ker(End(j_{!*}N)\rightarrow End(N))$ is nilpotent. 
	
	Hence $(1-r)^n = 0$ for some $n$, and defining $s$ by the relation $(1-r)^n = 1 -sr$ we see that $sr=1$. In particular the composite $j_{!*}N\overset{\bar\alpha}\longrightarrow M\overset{s\circ\bar\beta}\longrightarrow j_{!*}N$ is identity, and expresses $j_{!*}N$ as a summand of $M$.
\end{proof}

\begin{para}[Motivic intersection complex]\label{richerIC}
	We can replace condition $(2)$ by the stronger condition of \cite{wildeshaus_ic}:
	\begin{itemize}
		\item[(2$'$)] There is an isomorphism $j^*j_{!*}N \iso N$. The induced map $End(j_{!*}N)\rightarrow End(N)$ is injective.
	\end{itemize}
	In fact, we can work with a seemingly stronger condition
	\begin{itemize}
		\item[(2$''$)] There is an isomorphism $j^*j_{!*}N \iso N$. The natural map $End(j_{!*}N)\rightarrow End(N)$ admits a section. 
	\end{itemize}
	but which is equivalent since, the functor $j^*$ is full on Chow motives (\cite[1.7]{wildeshaus_ic}). Thus in this case, the intermediate extension is well defined up to a \emph{unique} isomorphism. 
	
	In this article, our primary motivation is to construct the intermediate extension for this stronger condition.
\end{para}


Finally, we recall certain facts about weights on the usual intersection complex in sheaves:

\begin{proposition}[Weights on $IC_X(\mathcal L)$]\label{weightsOnIC:sheaves}
	Let $X$ be an irreducible scheme over $k$. Let $j:U\hookrightarrow X$ be a dense immersion with $U$ open regular. Let $N\in DM(U)$ be pure of weight $N$, that is $N\in CHM(U)[n]$. Further assume that the realization $r(N)\in D^b_c(U)$ is a local system. Assume that $j_{!*}N \in DM(X)$ as in \ref{def:motivicIC} exists, and assume that $IC_X(\mathcal L) = r(j_{!*}N)$ where $IC_X(\mathcal L)\in D^b_c(S)$ denotes the intersection complex with coefficients in $\mathcal L$. 
	
	Then for any $i:Z\hookrightarrow X$ closed immersion with $Z$ not the same as $X$, 
	\begin{align*}
		i^*IC_X(\mathcal L)\in {^wD^{\le n+\dim X-1}}(Z)& &i^!IC_X(\mathcal L)\in {^wD^{\ge n+\dim X+1}}(Z).
	\end{align*}
\end{proposition}
\begin{proof}
	The weights on $i^*r(A)$ (resp. $i^!r(A)$) are by definition the weights on $r(i^*A)$ (resp. $r(i^!A)$), where $A=j_{!*}N$. $IC_X(\mathcal L)$ is perverse (upto a shift by $\dim X$). Further since $j_{!*}N$ is pure of weight $n$, it follows that $IC_X(\mathcal L)$ is also pure of weight $n$ by definition. In summary
	\[
		^pH^i(IC_X(\mathcal L)) \text{ vanishes for } i\ne\dim X \text{ and is pure of weight }n+i.
	\]
	We know that for the intersection complex, $^pH^j(i^*IC_X(\mathcal L)) =0$ for any $j\ge \dim X$, while $^pH^j(i^!IC_X(\mathcal L)) \le 0$ (see \cite[2.1.9]{BBD}). 
	Since $i^*$ decreases (Bondarko) weights (resp., $i^!$ increases Bondarko weights) while $i_*$  preserves weights, it follows that
	\begin{align*}
		^pH^i(i^*IC_X(\mathcal L)) &\text{ has weights }\le n+i \le n+\dim X-1\\
		{^pH^i}(i^!IC_X(\mathcal L)) &\text{ has weights }\ge n+i \ge n+\dim X+1
	\end{align*}
	and hence the statement follows by definition of $({^wD}^{\le n}(X), {^wD}^{\ge n}(X))$ \ref{def:weightTruncationInRealization}.
\end{proof}

Following lemma will be useful:
\begin{lemma}\label{jpreservesweights}
	Let $A\in DM(X)$ such that $r(A)\in {^wD^{\le n}}(X)$ (resp. $r(A)\in {^wD^{> n}}(X)$). Then for any open immersion $j:U\hookrightarrow A$, $r(j^*(A))\in {^wD^{\le n}}(U)$ (resp. $r(j^*(A))\in {^wD^{> n}}(U)$). 
\end{lemma}
\begin{proof}
	A weight filtration of $^pH^i(r(A))$ restricts to a weight filtration of $^pH^i(r(j^*A))$ since $j^*$ is exact for Bondarko weights as well as commutes with $^pH^i$.
\end{proof}

\section{Main results} \label{sec:mainresults}

\subsection{Weight truncations over a field} \label{sec:overField}

Let $DM(k)$ denote the rigid symmetric monoidal triangulated category of motives over $k$ with $\Q$ coefficients as in \S \ref{sec:newmotives}. We will be particularly interested in motives of Abelian varieties over $k$. 

\begin{definition} For $i\in\Z$, we make the following definitions:
\begin{align*}
	S_d(k)	&:= \{h(X)(-j)\in DM(k)\big| X/k\text{ an abelian variety}, \dim X\le d\} \\
	S_d^{\le i}(k)	&:= \{h^{l}(X)(-j)\in DM(k)\big| X/k\text{ an abelian variety}, \dim X\le d, 0\le l\le 2d, (l+2j)\le i\} \\
	S_d^{> i}(k)	&:= \{h^{l}(X)(-j)\in DM(k)\big| X/k\text{ an abelian variety}, \dim X\le d, 0\le l\le 2d, (l+2j)> i\}\\
	{^wDM_{d}^{Ab}}(k) &:= \vvspan{S_d(k)}\hspace{3em} {^wDM_{d}^{Ab\le i}}(k) := \vvspan{S_d^{\le i}(k)} \hspace{3em} {^wDM_{d}^{Ab> i}}(k) := \vvspan{S_d^{> i}(k)}
\end{align*}
where the generation is inside the triangulated category $DM(k)$. Note that we are not insisting $j$ to be non-negative.
\end{definition}
 We are only interested in the case $d=2$. It is on the category $DM_2^{Ab}(k)$ (that is the dimension of Abelian varieties that appear are $\le 2$) that we will be able to extend the functors $h^{\le i}$ and $h^{>i}$. The extended functor will be the motivic analogue of the Morel's weight truncations \cite{morelThesis} over a field and hence truncations for the same will be denoted as $w_{\le i}$. 

\begin{theorem}\label{tstruct:DMAb}
	For all $i\ge 0$, $(DM_2^{Ab\le i}(k), DM_2^{Ab>i}(k))$ forms a $t$-structure on $DM_2^{Ab}(k)$.
\end{theorem}
\begin{proof}
	For $A=h(X)(-j)\in S_d(k)$, let $r\ge -1$ be largest integer such that $r+2j\le i$. Then:
	\[
		h^{\le r}(X)(-j)	\rightarrow A	\rightarrow h^{> r}(X)(-j)\rightarrow 
	\]
	is a distinguished triangle with first term in $DM_2^{Ab\le i}(k)$ and last in $DM_2^{Ab>i}(k)$ by definition. This gives a decomposition for each object of $S_d(k)$. Since $DM_d^{Ab}(k)=\vvspan{S_d(k)}$, by \ref{tFromGenerators} it is enough to show that $\hom(A,B[m])=0$ for all $A\in S_d^{\le i}$ and $B\in S_d^{> i}(k)$ and any $m\in \Z$.
	
	Therefore we want to show that, for all $l,l',j,j',m\in \Z$ we have  
	\[
		\hom (h^l(X)(-j), h^{l'}(X)(-j')[m])=0\text{ whenever } (l'+2j')-(l+2j)>0 \text{ (eqiv. }j'-j>\frac{l-l'}{2}\text{)}
	\]
	Now notice that we need $0\le l,l'\le 4$ for the quantities to be not zero. Furthermore, we have that $h^3(X)=h^1(X)(-1)[-2]$, and $h^4(X)=h^0(X)(-2)[-4]$, therefore we can even assume that $0\le l,l'\le 2$. Define $r:=j'-j$. Then we have to show the vanishings in the following cases
	\begin{align*}
		(l,l')\in \{(0,1), (1,2), (0,2)\} &\text{ and } r=j'-j\ge 0	\\
		(l,l')\in \{(0,0), (1,1), (2,2), (1,0), (2,1)\} &\text{ and } r=j'-j\ge 1 \\
		(l,l')\in \{(2,0)\} &\text{ and } r=j'-j\ge 2
	\end{align*}
	Since Tate twist, $A\mapsto A(1)$, is an equivalence of categories, $\hom (h^l(X)(-j), h^{l'}(X)(-j')[m]) = \hom (h^l(X), h^{l'}(X)(-r)[m])$. Therefore the claim will follow from the vanishings of (for $r,m\in \Z, r\ge 0$):	
	\begin{align*}
		\hom(h^0(X), h^{\ge 1}(X)(-r)[m]) 	&=0 	& \hom(h^0(X), h(X)(-r-1)[m]) &= 0\\
		\hom(h^1(X), h^{\ge 2}(X)(-r)[m]) 	&=0 	& \hom(h^1(X), h(X)(-r-1)[m] &=0\\
		\hom(h^2(X), h^{\ge 3}(X)(-r)[m]) 	&=0 	& \hom(h^2(X), h^{\ge 1}(X)(-r-1)[m]) &= 0\\
		\hom(h^2(X), h(X)(-r-2)[m])			&=0 	& \hom(h^0(X), h^{\ge 2}(X)(-r+1)[m]) &=0		
	\end{align*}
 	These vanishings were proved in \cite[3.2.5]{vaish2017weight} (for $X$ an arbitrary surface, not merely an Abelian variety, where we use \ref{h:oldisnew} to identify the two).
\end{proof}

\begin{remark}
	In fact the same calculations show that the corresponding $t$-structure (for arbitrary $F$) can be extended to the smallest triangulated subcategory generated by motives of surfaces (and not merely Abelian surfaces) up to arbitrary Tate twists. This is a mild generalization of the key result in \cite{vaish2017weight}.
\end{remark}

\subsection{Spreading out}\label{sec:spread} Let us fix a base field $k$. For any finitely generated field $K/k$ let $t(K)$ denote the transcendence degree of $K/k$ (typically $K$ will be the function field of a variety of finite type over $k$ and then $t(K)$ measures the dimension of the variety). Fix $X$ an irreducible variety of finite type over $k$. 

\begin{definition}\label{define:motivicMorelTStructure}
	Fix a function $F:\Z_{\ge 0}\rightarrow \Z$ to be any monotone step function (i.e. $F$ is monotone and $0\le |F(x)-F(x-1)|\le 1$. It follows that for $x\ge y$
		\[
			-x+y\le F(x)-F(y)\le x-y
		\]
	which is what will be essentially used below. The definitions are inspired by \cite[\S 3.1]{arvindVaish}, but for the main application in this article, we could also work with just the constant function $\dim X$.
	
	 By $({^wDM^{\le F}(K)}, {^wDM^{> F}(K)})$ we mean the $t$-structure, $({DM_2^{Ab\le i_K}(K)}, {DM_2^{Ab> i_K}(K)})$ on the category $DM^{Ab}_2(K)$ where $i_K=F(t(K)) - t(K)$. For any $K$, this is a $t$-structure by \ref{tstruct:DMAb}, and hence by \ref{gluing:mainresult} induces a $t$-structure on a suitable full subcategory $D'(X)\subset DM(X)$.
\end{definition}

\begin{lemma}\label{spreading:truncation}
	Let $A\in DM(X)$ such that for the generic point $\eta:\Spec K\hookrightarrow X$, $\eta^*(A)\in DM^{Ab}_2(K)$ (resp. $\eta^*(A)\in DM^{Ab, \le F}(K)$, resp. $\eta^*(A)\in DM^{Ab, >F}(K)$). Then there is a neighbourhood $j:U\subset X$, such that $j^*A\in D'(U)$ (resp. $j^*A\in {^wDM^{\le F}}(U)$, resp. $j^*A\in {^wDM^{>F}}(U)$).
\end{lemma}
\begin{proof}
	Let $\eta^*A\in DM^{Ab,> F}(K) = \vvspan{S^{> i_K}_2(K)}$ for $i_K = F(t(K)) - t(K)$. Since objects in $\vvspan{S^{> i_K}_2(K)}$ are obtained by finitely many iterations of shifts, cones, and taking summands of objects in $S_2^{> i_K}(K)$, and $\eta^*$ preserves these operations, we can restrict to the case $\eta^*A\in S^{> i_K}_2(K)$. 
	
	Hence assume that $\eta^*A = h^k(Y_\eta)(-j)$ for $k+2j> i_K = F(t(K)) - t(K)$ where $\pi:Y_\eta \rightarrow \Spec K$ is an abelian variety of dimension $\le 2$. Spreading out, we can find an abelian scheme over $Y_U\rightarrow U$ over some open dense $U$ in $X$ such that $(Y_U)|_\eta = Y_\eta$. By restricting $U$ if needed, we can also assume that $U$ is regular.
	
	Consider $B := h^k(Y_U)(-j)\in DM(U)$. Then for any $\epsilon:\Spec L\hookrightarrow U$, we have that $\epsilon^*B := \epsilon^*(h^k(Y_U)(-j)) = h^k(\epsilon^*Y_U)(-j)= h^k(Y_\epsilon)(-j)$. Also, since $Y_U/U$ is smooth $Y_U, Y_\epsilon$ are regular and using purity and definition of $\epsilon^!$ it follows that $\epsilon^!(B) = B(-d)[-2d]$ where $d=(t(K)-t(L))$.
	
	Since $Y_\epsilon:=(Y_U)|_\epsilon$ is also an abelian scheme, and 
	\[
		k+2j + 2(t(K)-t(L))>F(t(K))- t(K) + 2(t(K) - 2t(L)) \ge F(t(L)) - t(L)
	\]
	hence $\epsilon^*B\in DM^{Ab, > F}(L)$. Since $\eta^*B\cong \eta^*B$, by shrinking $U$ further if needed, using continuity, we can assume that $B=j^*A$. The claim is now immediate. 
	
	The calculation for $\eta^*A\in DM^{Ab,\le F}(K)$ is ditto and even simpler.
	
	For $\eta^*A \in DM^{Ab}_2(K)$, let $\hat B = w_{\le F}(\eta^*a)$, the truncation for the $t$-structure for $F$ on $DM^{Ab}_2(K)$, and let $\hat g:\hat B\rightarrow \eta^*B$ denote the natural morphism. Then, by continuity, $\hat g=\eta^*g$ for some morphism $g:B\rightarrow j^*A$ on a neighbourhood $U$, such that $\hat B = \eta^*B$. By restricting $U$ further, by what we saw above, we can assume that $B\in {^wDM^{\le F}}(U)$. Let $C = Cone(h)$. Then, by definition of $\hat g$, we have $\eta^*C = w_{>F}\eta^*A$. Hence by what we saw above, and restricting $U$ further if needed, we can assume that $C\in {^wDM^{> F}}(U)$. Hence, by definition, $j^*A\in D'(U)$ as required.
\end{proof}

\begin{lemma}\label{weightsOfRealization}
	Let $S/k$ be an irreducible scheme and $F$ be a monotone step function. Let $A\in {^wDM^{\le F}}(S)$ (resp. $A\in {^wDM^{> F}}(S)$). Then there is a $j:U\subset S$ open dense, such that $r(j^*A))\in {^wD}^{\le F(t(K))}(U)$ (resp. $r(j^*A))\in {^wD}^{> F(t(K))}(U)$, where $\eta:\Spec K\hookrightarrow S$ is the generic point.
\end{lemma}
\begin{proof}
	Let $n=F(t(K))$. Let $A\in {^wDM^{\le F}}(S)$. Hence $\eta^*A\in DM_2^{Ab,\le n}(K)$. Hence by continuity it is enough to show that for each $A\in DM_2^{Ab,\le n}(K)$, there is an open $U\subset S$, and an object $\bar A\in DM(U)$, such that $r(\bar A)\in {^wD}^{\le n}(U)$. Since, the category ${^wD}^{\le n}(U)$ is stable under shifts, taking summands, or cones, therefore we can as well assume that $A\in S^{Ab,\le n}_2(K)$, in particular $A=h^i(X)(-j)$ with $i+2j\le n$ and $X$ an abelian scheme over $K$. By spreading out, we can find an abelian variety, $\bar X$ over some $U\subset S$ open dense, $U$ regular, and then we can consider $\bar A=h^i(\bar X)(-j)$. Using that $h^i$ is the $n^i$ eigenspace under multiplication by $n$, it is easy to see that
	\[
		r(h^i(\bar X))\subset \mathcal H^i(\bar X)[-i]\in {^wD}^{\le i}(S)\text{ and hence }r(h^i(\bar X)(-j)) \in {^wD}^{\le i+2j}\subset {^wD}^{\le n}
	\]
	 since $i+2j\le n$, as required. 
	 
	 The case when $A\in {^wDM^{> F}}(S)$ is similar. 
\end{proof}

A converse of the above also holds:

\begin{lemma}\label{spreading:realization}
	Let $A\in DM(X)$ such that for $\eta:\Spec K\rightarrow X$ generic, $\eta^*A\in DM^{Ab}_2(K)$. Assume that for some open set $f:V\hookrightarrow X$, $r(f^*A)\in {^wD}^{\le F(t(K))}(X)$ (resp. $r(f^*A)\in {^wD}^{>F(t(K))}(X)$). Then there is a $U\subset X$ open dense such that $j^*A\in DM^{Ab, \le F}(U)$ (resp. $j^*A\in DM^{Ab, >F}(U)$). 
\end{lemma}
\begin{proof}
	Let $n=F(t(K))$. Assume $r(A)\in {^wD}^{\le n}(X)$, the other case is similar. By \ref{spreading:truncation}, it will be enough to show that $\eta^*A\in DM^{Ab, \le n}(K)$, and hence enough to show that $j^*A\in DM^{Ab, \le n}(U)$. Hence replacing $X$ by $V$, we may as well assume that $F=n$ is the constant function.
	
	 Again, by \ref{spreading:truncation}, there is a $U\subset X$ such that $j^*A \in DM'(U)$. Hence we can write a triangle:
		\[
			w_{\le F}j^*A \rightarrow j^*A \rightarrow w_{> F}j^*A\rightarrow 
		\]
	Also by \ref{jpreservesweights} $^pH^i r(j^*A)$ has weights $\le F$.
	
	By \ref{weightsOfRealization}, restricting $U$ if needed, $^pH^i (r(w_{\le F}j^*A))$ has weights $\le F$ while $^pH^i (r(w_{> F}j^*A))$ has weights $>F$. Also, we have a long exact sequence of cohomology:
		\[
			\rightarrow {^pH^i} (r(j^*A)) \rightarrow {^pH^i} (r(w_{> F}j^*A))  \rightarrow {^pH^{i+1}} (r(w_{\le F}j^*A)) \rightarrow 
		\]
		but this forces that ${^pH^i} (r(w_{> F}j^*A))$ is of weights $\le F$ as well. In particular, it forces that ${^pH^i} (r(w_{> F}j^*A)) = 0$ for all $i$ and hence $r(w_{> F}j^*A) = 0$. 
		
		By restricting $U$ if needed and using continuity we can assume that $j^*A\in DM^{Ab}(U)$, since the same holds at the generic point. 
		
		We claim that $w_{>F}j^*A\in DM^{Ab}(X)$ as well. Now $j^*A$ is obtained by taking finitely many shifts, cones, and summands of motives $h_U(X)(-j)$ where $X/U$ is an abelian scheme. Since $w_{>F}$ preserves these operations, it is enough to verify the same for $A=h_U(X)(-j)$, in which case $w_{>F}$ is just a summand of $A$ and hence in $DM^{Ab}(X)$.
		
		Hence by conservativity \ref{conservativityonab} it follows that $w_{>F}j^*A = 0$. It follows that $j^*A\cong w_{\le F}j^*A$ as required.
\end{proof}

\begin{notation}\label{def:abtype}
	Call a proper map $\pi:Y\rightarrow X$ to be of \emph{abelian type} if for any (Zariski) point $x\hookrightarrow X$, the fibre $Y_x = \cup_{i\in I} Z_i$ in irreducible components (where $I$ is a finite indexing set), and such that for each $J\subset I$
	\begin{align*}
		Z_J:=\cap_{i\in J}Z_i\text{ is smooth s.t. }(\pi|_{Z_J})_*1_{Z_J}\in DM^{Ab}(x)
	\end{align*}
	If in addition, $(\pi|_{Z_J})_*1_{Z_J}\in DM^{Ab}_d(x)$, we say it is of \emph{abelian type of relative dimension $\le d$}. We are going to be particularly interested in the case $d=2$.
\end{notation}

\begin{lemma}\label{abelianTypeRestrictWell}
	Let $\pi:Y\rightarrow X$ be a proper map with $Y$ smooth. Let $\epsilon:\Spec L\rightarrow X$ be any Zariski point of $X$ such that $\pi |_{\pi^{-1}(\bar \epsilon)}$ is of abelian type of relative dimension $\le d$, where $\bar \epsilon$ is the Zariski closure of $\Spec L$. Then $\epsilon^*(\pi_*1_Y)$ and $\epsilon^!(\pi_*1_Y)$ lie in the category $DM^{Ab}_d(L)$. 
\end{lemma}
\begin{proof}
	Let $Y_\epsilon = \cup_{i\in I} Z_i$ and as before, let $Z_J:=\cap_{i\in J}Z_i$ for any $J\subset I$ finite. For simplicity of notation, in below we will let $\pi$ also denote projection from any subset of $Y$ to $X$. 

	Since smoothness is an open condition, we can replace $\epsilon$ by a neighbourhood $V\subset \bar \epsilon$ such the inclusion $V\hookrightarrow X$ is of finite type, and such that the pre-image $Y_V = \cup_{i\in I}Z'_i$ with $Z_J':=\cap_{i\in J}Z'_j$ is smooth, proper over $V$ for any subset $J\subset I$.
	Let $f:Y_V\hookrightarrow Y$ denote the inclusion of the normal crossing variety $Y_V$ (that is such that each irreducible component of $Y_V$ as well as their $k$-fold intersection is smooth). By base change we have
	\begin{align*}
		\epsilon^*(\pi_*1_Y) = \pi_* \epsilon^*f^*1_Y = \epsilon^*1_{Y_U} & & \epsilon^!(\pi_*1_Y) = \pi_*\epsilon^*f^!1_Y
	\end{align*}
	We will induct on the largest $k$ such that some $k$-fold intersection of $Z'_j$ is non-trivial (where we assume that $Y_V\subset Y$ is normal crossing and each $k$-fold intersection is proper and smooth over $V$) and show that $\epsilon^*\pi_*f^*1_Y$ (resp. $\epsilon^*\pi_*f^!1_Y$) is in $DM^{Ab}_d(L)$ as required. 
	
	\emph{(Base case)}: If $k=1$, then $Y_V$ is a disjoint union of smooth varieties $Z'_i$ and hence is smooth. Hence $f^!1_Y = 1_{Y_U}(-c)[-2c]$. Also $\pi_*\epsilon^*1_{Y_U} = \oplus_{i}\pi_*1_{Z_i}$  and hence the claim follows directly from assumption. 
	
	\emph{(Induction Step)}: Let $i:Y'_V\hookrightarrow Y_V$ be the closed subset formed by intersection of components inside $Y_V$ and let $j:W\hookrightarrow Y_V$ denote the open complement. Hence we have the exact triangles:
	\begin{align*}
		j_!j^*f^*1_Y\rightarrow f^*1_Y\rightarrow i_*i^*f^*1_Y\rightarrow& &i_*i^!f^!1_Y\rightarrow f^!1_Y\rightarrow j_*j^*f^!1_Y\rightarrow
	\end{align*}
	By induction, $\epsilon^*\pi_*i_*i^*f^*1_Y$ (resp. $\epsilon^*\pi_*i_*i^!f^!1_Y$) lies in $DM^{Ab}_d(L)$ (since we are then working with normal crossing variety $Y'_V$ which has at most $(k-1)$-fold intersections) and it is enough to show that $\epsilon^*\pi_*j_!j^*f^*1_Y$ (resp. $\epsilon^*\pi_*j_*j^*f^*1_Y$) does too. 
	
	But $W$ is a disjoint union $W=\sqcup_{i\in I}W_i$ with $j_i:W_i\hookrightarrow Z'_i$ open dense. Hence, we can as well assume that $W=W_i$, and work with one $i$ at a time, that is it is enough to show that $\epsilon^*\pi_*j_{i!}j_i^*f_i^*1_Y, \epsilon^*\pi_*j_{i*}j_i^*f_i^!1_Y \in DM^{Ab}_d(L)$, where $f_i:Z'_i\hookrightarrow Y$ denotes the immersion. 
	
	Let $g_i:T_i\hookrightarrow Z'_i$ denote the complement of $W_i$. We have triangles:
	\begin{align*}
		j_{i!}j_i^*f_i^*1_Y\rightarrow f_i^*1_Y\rightarrow g_{i*}g_i^*f_i^*1_Y\rightarrow& &g_{i*}g_i^!f_i^!1_Y\rightarrow f_i^!1_Y\rightarrow j_{i*}j_i^*f_i^!1_Y\rightarrow
	\end{align*}
	Now we apply $\epsilon^*\pi_*$ to the triangle. But $f_i$ is inclusion of a smooth variety and $g_i\circ f_i$ is inclusion of $T_i$ which is a normal crossing variety, with components as some sub-collection of $Z'_J$ and which has at most $(k-1)$-fold intersections. Hence it follows that 
	\[
		\epsilon^*\pi_*f_i^*1_Y, \epsilon^*\pi_*f_i^!1_Y, \epsilon^*\pi_*g_{i*}g_i^*f_i^*1_Y, \epsilon^*\pi_*g_{i*}g_i^!f_i^!1_Y \in DM^{Ab}_d(L)
	\]
	by induction. It follows that
	\[
		\epsilon^*\pi_*j_{i!}j_i^*f_i^*1_Y, \epsilon^*\pi_*j_{i*}j_i^*f_i^!1_Y \in DM^{Ab}_d(L)
	\]
	as well, as required.
\end{proof}

\begin{theorem}\label{tech:mainresult}
	Let $X$ be a variety and $\pi:Y\rightarrow X$ be a proper map with $Y$ smooth. Let $V\subset X$ be smooth such that $\pi:{\pi^{-1}V}\rightarrow V$ is an abelian scheme and $\pi|_{\pi^{-1}(X-V)}$ is of abelian type of relative dimension $\le 2$. Let $N$ be a summand of $h^n(X)$ such that the realization $r(N[-n])$ is a local system on $X(\C)$. Assume that the intermediate extension $j_{!*}N$ exists in the weaker sense of \ref{def:motivicIC}. Further assume that $r(j_{!*}N) = IC_X(\mathcal L)[n]$, the usual intersection complex (see \cite{BBD}) with coefficients in $\mathcal L$ (upto a shift).
	
	 Then $j_{!*}N$ exists in the stronger sense of \ref{richerIC}.
\end{theorem}
\begin{proof}
	Let $d=\dim X$. Let $j_V:V\subset X$ be regular such that if $Y_V:=\pi^{-1}V, \pi_{Y_V}$ is smooth, proper. By definition $\pi_*1_Y$ is of weight $0$ in $DM(X)$. Since $N$ is a summand of $j_V^*(\pi_*1_Y)$, it follows by \ref{ICissummandofchow} $j_{!*}N$ is a summand of $\pi_*1_Y$. There is a triangle:
		\begin{equation}\label{eq:IM}
			j_{!*}N	\rightarrow j_{V*}N	\rightarrow B\rightarrow 
		\end{equation}
	Therefore we only need to show that the natural map $End(j_{!*}N)\rightarrow End(N)$ is an isomorphism. Now, since $j_*$ is functorial, it will be enough to show that
	\begin{align*}
		\hom(j_{!*}N, B) = 0& &\hom(j_{!*}N, B[-1]) = 0
	\end{align*}
	
\newcommand{\dminus}{d}
	Let $i_V:(X-V)\hookrightarrow X$.
	We claim that $i_V^*j_{!*}N \in {^wDM^{\le \dminus+n}}(X-V)$ while $B = i_{V*}C$ and $C \in {^wDM^{>\dminus+n}}(X-V)$. It follows that $i_V^*j_{!*}N = w_{\le \dminus+n}i_V^*j_{V*}N$, the truncation for the $t$-structure in \ref{define:motivicMorelTStructure} (for the constant function $\dminus+n$) and the result would then follow from the vanishing
	\[
		\hom(j_{!*}N, B[i]) = \hom(j_{!*}N, i_{V*}C[i]) = \hom(i_V^*j_{!*}N, C[i]) = 0, \forall i\in \Z
	\]
	
	\emph{To show that $i_V^*j_{!*}N\in {^wDM^{\le \dminus+n}}$:}
		Let $A:=j_{!*}N$. Let $j:U\subset X$ be an open set such that $i^*j^*A \in {^wDM^{\le \dminus+n}}(U\cap (X-V))$ where $i:U\cap (X-V)\hookrightarrow U$ denotes the natural immersion. We will do a Noetherian induction on the complement $i:Z:=X-U\hookrightarrow X$ (starting with $Z=X-V$).

		
		We saw that $A$ is a summand of $\pi_*1_Y$. Hence $i^*(A)$ is a summand of $i^*\pi_*1_Y$. Therefore by \ref{abelianTypeRestrictWell} $\eta^*(A) \in DM^{coh, Ab}_2(\eta)$ for any generic point $\eta\in Z$, and hence by \ref{spreading:truncation} there is a $j_W:W\subset Z$ open such that $j_W^*(A)\in DM^{Ab}_2(W)$.
		
		By \ref{weightsOnIC:sheaves}, $r(i^*j_W^*(A)) \in {^wD^{\le \dminus + n}}$. Hence, by shrinking $W$ if necessary, by \ref{spreading:realization}, $i^*j_W^*(A)\in {^wDM^{\le \dminus + n}}$. Therefore by definition $j':U':=X-(Z-W) \subset X$ open, $i^*j^{'*}IC_X \in {^wDM^{\le \dminus+n}}(U')$ and we are done since $Z-W$ is a proper subset of $Z$.
		
	\emph{To show that $B =i_*C$:}
	 We have $j_V^*(B)=0\in {^wDM^{> \dminus+n}}(V)$ since $j^*_Vj_{!*}N\cong N$. Hence $B=i_{V*}i_V^*B$ using localization and we let $C=i_V^*B$. 
	
	\emph{To show that $C\in {^wDM^{> \dminus+n}}$:}
		Let $j:U\subset X-V$ be an open set such that $j^*C \in {^wDM^{>\dminus+n}}(U)$ we will do a Noetherian induction on the complement $i:Z:=(X-V)-U\hookrightarrow X-V$ as before (beginning with $Z=X-V$ as before).

		Now $A$ is a summand of $\pi_*1_Y$. 
		By \eqref{eq:IM}, $C=i_V^{!}B = {i_V^!}A[1]$ is a summand of $i_V^{!}\pi_*1_Y[1]$. Since $Z\subset X - V$ we conclude that $i^!C$ is a summand of $i^!\pi_*1_Y[1]$. Now the proof works using realizations and conservativity, ditto as it did for $A$.
\end{proof}
\begin{remark}\label{DcouldBeSmaller}
	We are not using \ref{weightsOnIC:sheaves} in it's full strength in above. It is clear from the proof that instead of the constant function $d+n$ we could have used any monotone step function $F$ with
	\begin{align*}
		F(d) = d+n& &d+n-1\le F(i)\le d+n+1\text{ for }i<d.
	\end{align*}
	and this is consistent with the results in mixed sheaves (see, for example, \cite[4.2.4]{vaish2017weight}).
\end{remark}


\subsection{Shimura varieties}\label{sec:shimura} We summarize the consequences of our main result in the context of Shimura varieties.

In below we summarize the key points of the discussion in \cite[\S 8]{wildeshaus_shimura_2012}:
\begin{para}\label{shimura:situation}
	Assume that $(P,\mathfrak X)$ is a mixed Shimura datum (in particular $G$ is a connected linear algebraic group over $\Q$). Then, associated to any open compact subgroup $K\subset P(\mathbb A_f)$ there is an associated quasi projective varieties $M^K:=M^K(P,\mathfrak X)$ whose complex points can be identified with:
	\[
		M^K(\C) = P(\Q)\backslash (\mathfrak X\times P(\mathbb A_f))/K).
	\]
	In fact the variety $M^K$ can be defined over the reflex field $E(P,\mathfrak X)$ which is a number field in general. For simplicity we will assume $K$ to be neat. 
	
	Following \cite[\S 8]{wildeshaus_shimura_2012} we assume the following condition on $(P,\mathfrak X)$:
	\begin{itemize}
		\item[(+)] If $G$ is the maximal reductive quotient of $P$, then the connected component of identity $Z(G)^0$ of the centre $Z(G)$ is, up to isogeny, product of a $\Q$ split torus with a torus of compact type (that is whose $\R$ points are compact). 
	\end{itemize}
	
	Let $(G,\mathfrak h)$ is the pure Shimura datum underlying $(P,\mathfrak X)$ (that is $(P,\mathfrak X)/W = (G,\mathfrak h)$ in the sense of \cite[2.9]{pink1990arithmetical}) and fix $L\subset G(\mathbb A_f)$ compact open. Then $M^L(G,\mathfrak h)$ admits a Baily-Borel compactification, $(M^L)^*$ which comes with a natural stratification (see \cite[Chap 6]{pink1990arithmetical}). 
	
	Furthermore, $M^K(P,\mathfrak X)$ comes equipped with a Toroidal compactification $(M^K)^\Sigma$, for a good choice of ``cone decomposition'' $\Sigma$ and this also comes with a natural stratification (\cite[Chap 7]{pink1990arithmetical}). 
	
	The morphism of Shimura data $(P,\mathfrak X)\rightarrow (G,\mathfrak h)$ gives rise to a map $M^K\rightarrow M^L$ which in turn extends to a map 
	\[
		\pi: (M^K)^\Sigma\rightarrow (M^L)^*
	\]
	There are good conditions on $\Sigma$ such that the map is projective. Also, the map is stratified for natural stratifications on the two strata.
	
	By \cite[8.4]{wildeshaus_shimura_2012}, the natural map between the stratifications factorizes as:
	\[
		\pi|_T:T\overset{\pi''}\longrightarrow B\overset{\pi'}\longrightarrow S
	\]
	where $S\subset (M^L)^*$ and $T\subset (M^K)^\Sigma$ are strata under the natural stratification above, with $\pi''_*1_T$ a mixed Tate motive over $B$ (that is in the category $DMT(B)$ in the sense of \cite[4.3]{wildeshaus_shimura_2012}), $\pi'$ is proper smooth and fibrewise, on geometric points of $S$, $B/S$ is isomorphic to a disjoint union of abelian varieties.
	Under assumption $(+)$, $T$ is the variety associated to certain Shimura datum, while $S$ is a finite quotient of the variety associated to a pure Shimura datum. Both $T$ and $S$ are defined over the reflex field associated to $(P,\mathfrak X)$ and are smooth. 
	
	In particular the map $\pi$ is of Abelian type. Furthermore it is of relative dimension $\le d$ if the relative dimension of $B/S$ is $\le d$.
\end{para}

\begin{para}\label{notn:B}
	We can compute dimension of $B$ as follows. Fix a proper admissible subgroup $Q$ of $P$ with associated normal subgroup $P_1$ as in \cite[4.7]{pink1990arithmetical}. Let $W_1$ be unipotent radical of $P_1$, $U_1\subset W_1$ denote the centre of $W_1$ (that is, it is the ``weight $(-2)$'' part of $P_1$). 
	
	To each stratum $T$ we can associate $\sigma\times {p}$ with $\sigma$ a rational polyhedral cone inside $U_1(\R)(-1)$ and $p\in P(\mathbb A_f)$. In particular, associated to each such $\sigma$, we can find an algebraic subgroup $\langle \sigma \rangle \subset U_1$ and then the group appearing in the mixed Shimura datum corresponding to $T$ is $P_{1,[\sigma]}:=P_1/\langle \sigma \rangle$.  It's unipotent radical is $W_{1,[\sigma]}=W_1/\langle\sigma\rangle$ and it's ``weight $(-2)$'' part is is $U_{1,[\sigma]}:=U_1/\langle\sigma\rangle$. 
	
	Then (up to a further finite quotient) the map $\pi|_T$ corresponds to map of the mixed Shimura variety to it's pure part and hence relative dimension of $B/S$ is the dimension of $V_1 := W_{1,[\sigma]}/U_{1,[\sigma]} = W_1/U_1$ (complex dimension, that is dimension of $\frac 1 2\dim_{\R} V_1(\R)$).
\end{para}

Then the following result is an immediate corollary of \cite[0.1 and 0.2]{wildeshaus_shimura_2012} and \ref{tech:mainresult}:
\begin{theorem}\label{shimura:tech:mainresult}
	Let $(G,\mathfrak X)$ be a pure Shimura datum and follow the notation of \ref{notn:B}. If for any proper admissible subgroup $Q$, the corresponding $V_1$ has complex dimension $\le 2$, then there is a motivic intersection complex $IM_X$ in the stronger sense of \ref{richerIC} for $X=M_K^*$, the Baily-Borel compactification of corresponding Shimura variety. 
\end{theorem}

In special cases we have the following:

\begin{corollary}\label{shimura3folds}
	Let $(G,\mathfrak X)$ is a pure Shimura datum and such that the dimension of the corresponding Shimura variety $M^K$ (for any fixed open compact $K\subset G(\mathbb A_f)$) is $\le 3$, then the intersection motive corresponding to the Baily-Borel compactification $(M^K)^*$ of $M^K$ exists in the sense of \ref{richerIC}.
\end{corollary}
\begin{proof}
	In this case dimension of the boundary itself for any resolution is $\le 2$, and hence so is the corresponding abelian part. Thus we are in the situation of \ref{shimura:tech:mainresult}.
\end{proof}

\begin{remark}
	It is clear from \ref{h:oldisnew} that the construction here coincides with the construction in \cite[\S 4.2]{vaish2017weight} (using \ref{DcouldBeSmaller} to compare with $F$ in \cite[4.2.2]{vaish2017weight}) and hence we have shown that, in the case of Baily-Borel compactification of Shimura three-folds, the motivic intersection complex constructed in \cite[4.2.2]{vaish2017weight} coincides with the one constructed in \ref{shimura3folds} and is actually an intermediate extension in the stronger sense of \ref{richerIC}.
\end{remark}

\begin{remark}\label{result:4folds}
	The constructions here will also be valid for several Shimura varieties of dimension $4$ -- for any resolution of singularities only fibres over zero dimensional strata can have dimension $3$. Thus the methods here are not applicable only if the fibres over zero dimensional strata are purely abelian. This does happen for the group $U(n,1)$ (see \cite[12.22]{pink1990arithmetical}) but this should be essentially the only example in dimension four (when $G^{der}$ is assumed to be absolutely simple).
\end{remark}

It is possible to work with certain higher dimensional Shimura varieties as well. For example, for the Siegel six fold, the abelian variety pieces that appear in the boundary of a toroidal resolution are of dimension $\le 2$ and hence \ref{shimura:tech:mainresult} gives: 

\begin{corollary}\label{siegel6}
	Let the Shimura datum correspond to the Siegel six fold, that is let $G = Sp(6)$ and let $\mathfrak X=\{ A \text{ a }3\times 3\text{ complex matrix }\big| A^t=A, \text{Im}\;A\text{ is positive definite}\}$.
	 Then for any open compact $K\subset G(\mathbb A_f)$, the  Baily-Borel compactification $(M^K)^*$ for the corresponding variety $M^K$ has the motivic intersection complex in the stronger sense of \ref{richerIC}. 
\end{corollary}

In some special cases it is also possible to construct intermediate extension of local systems. For example, using \cite[0.3]{wildeshaus_shimura_2012} and \ref{tech:mainresult}, and using that the fibres in the boundary of a toroidal compactification of the universal Abelian scheme are of dimension strictly less than the relative dimension of the universal Abelian scheme, we have the corollary:

\begin{corollary}\label{siegel6local}
	 Assume the situation of \ref{siegel6}. Further assume that $\pi:A\rightarrow M^K$ denotes the universal abelian scheme (which exists for good choices of $K$). Then for any summand $N$ of $h^n(A)$ for any $n$ such that $r(N)[-n]$ is a local system on $M^K(\C)$, $j_{!*}N$ is defined in the stronger sense of \ref{richerIC}. 
\end{corollary}

%
\bibliographystyle{alpha}
\bibliography{shimura_ic}{}
\end{document}